\documentclass[12pt]{amsart}
\usepackage{fullpage}
\usepackage{graphicx}
\usepackage{todonotes}

\usepackage{tikz}
\usetikzlibrary{decorations.markings}
\usepackage{tkz-euclide}
\newtheorem{theorem}{Theorem}[section]
\newtheorem{corollary}[theorem]{Corollary}
\newtheorem{lemma}[theorem]{Lemma}
\newtheorem{proposition}[theorem]{Proposition}

\numberwithin{equation}{section}
\newcommand\NN {{\mathbb N}}

\newcommand\QQ {{\mathbb Q}}
\newcommand\RR {{\mathbb R}}
\newcommand\TT {{\mathbb T}}
\newcommand\ZZ {{\mathbb Z}}

\newcommand\sltwoz{{\rm SL(2,\ZZ)}}

\newcommand\autgroup{{\rm Aut}}

\newcommand\id{{\rm Id}}

\newcommand\distance{{\rm Dist }}

\newcommand{\slope}{{\rm Slope }}

\newcommand\cA{{\mathcal{A}  }}

\newcommand\cE{{\mathcal{E}  }}

\newcommand\cK{{\mathcal{K}  }}

\newcommand\cO{{\mathcal{O}  }}

\newcommand\cQ{{\mathcal{Q}  }}

\newcommand\cT{{\mathcal{T}  }}

\begin{document}

%\subjclass[2010]{Primary 11J70, secondary 20H10}

%\keywords{Continued fractions, Fuchsian groups}

\title{A genus 4 origami with minimal hitting time and an intersection property}

\author{Luca Marchese}

\address{Dipartimento di Matematica, Universit\`a di Bologna, Piazza di Porta San Donato 5, 40126, Bologna, Italia}

\email{luca.marchese4@unibo.it}

%----------------------------------------------------------------

\begin{abstract}
In a minimal flow, the hitting time is the exponent of the power law, as $r$ goes to zero, for the time needed by orbits to become $r$-dense. 
We show that on the so-called \emph{Ornithorynque} origami the hitting time of the flow in an irrational slope equals the diophantine type of the slope. We give a general criterion for such equality. 
In general, for genus at least two, hitting time is strictly bigger than diophantine type.
\end{abstract}

\maketitle

%\tableofcontents

\section{Introduction}

An \emph{origami}, also known as \emph{square-tiled surface}, is a surface obtained glueing copies of the square $[0,1]^2$ along the boundaries. 
On a given origami, any $\alpha\in\RR$ defines a \emph{linear flow} in slope $\alpha$, whose dynamical properties are related to the diophantine properties of $\alpha$. This reflects a more general principle in \emph{Teichm\"uller dynamics}. \cite{ForniMatheusIntroduction} gives an introduction to the subject and a selection of the many relevant references. In this paper we consider a special genus $4$ origami called \emph{Ornithorynque} (see \S~\ref{SectionOrnithorynqueOrigami}). Our main Theorem~\ref{TheoremMainTheorem} states that on such origami the \emph{hitting time} in any slope $\alpha$ equals the diophantine type of $\alpha$. This is the minimal possible value for the hitting time (Lemma~\ref{LemmaLowerBoundHittingTime}), and in many cases the equality does not hold, according to~\cite{KimMarcheseMarmi}. We prove Theorem~\ref{TheoremMainTheorem} stating a general criterion based on a specific intersection property, namely Theorem~\ref{TheoremGeneralCriterion}, and showing that the Ornithorynque satisfies the intersection property (see \S~\ref{SectionIntersectionPropertyOrnithorynque}). This extends to the Ornithorynque results previously proved in~\cite{KimMarcheseMarmi} for the genus $3$ origami called \emph{Eierlgende Wollmilchsau}.

\subsection{Origamis and linear flows}
\label{SectionOrigamisAndLinearFlows}

Fix a finite set $\cQ$ and a pair $(h,v)$ of permutations of $\cQ$ generating a transitive subgroup 
$\langle h,v\rangle$ of the symmetric group. 
For any $j\in\cQ$ let $Q_j:=\{j\}\times[0,1]^2$ the $j$-th copy of the unit square. 
Denote by $l_j$, $r_j$, $b_j$, $t_j$ the copies of the four sides  
$$
l:=\{0\}\times[0,1]
\textrm{ , }
r:=\{1\}\times[0,1]
\textrm{ , }
b:=[0,1]\times\{0\}
\textrm{ , }
t:=[0,1]\times\{1\}.
$$
For any $j\in\cQ$ paste the right side $r_j$ of $Q_j$ to the left side $l_{h(j)}$ of $Q_{h(j)}$ and the top side $t_j$ of $Q_j$ to the bottom side $b_{v(j)}$ of $Q_{v(j)}$. An origami $X$ is a surface arising in this way. It is compact, connected, orientable and without boundary. 
We have a covering \footnote{Define it on $\cQ\times[0,1]^2$ as $\rho_X\big((j,x)\big):=[x]$. 
This gives a map on $X$ because glued points have the same image.}
\begin{equation}
\label{EquationCoveringOverTorus}
\rho_X:X\to\TT^2
\end{equation}
over the standard torus $\TT^2:=\RR^2/\ZZ^2$, ramified only over $[0]\in\TT^2$, where 
$[x]$ denotes the coset of $x\in\RR^2$. The points $p_1,\dots,p_m$ in $X$ where $\rho_X$ is ramified are in bijection with the cycles of the commutator $[v,h]:=v^{-1}h^{-1}vh$. Let $k_1,\dots,k_m$ in $\NN$ be such that for any $1\leq j\leq m$ the cycle of $[v,h]$ corresponding to $p_j$ has length $k_j+1$. 
The surface $X$ inherits a metric with a conical angle $2(k_j+1)\pi$ at any $p_j$ and which is flat outside these points. If $g$ is the genus of $X$, then $k_1+\dots+k_m=2g-2$. 
Details can be found in~\cite{ForniMatheusIntroduction}, while \S~\ref{SectionOrnithorynqueOrigami} below describes an explicit example. 

\smallskip

Fix $\alpha\in\RR\cup\{\infty\}$ and set $e_\alpha:=(\sin\theta,\cos\theta)$, where 
$
\theta:=\arctan{\alpha}\in(-\pi/2,\pi/2]
$, 
that is the unit vector $e_\alpha$ with slope $\alpha$. 
The \emph{linear flow} 
$
\phi_\alpha:\RR\times X\to X
$ 
on $X$ is the continuous flow determined for any $p\in X$ and $t\in\RR$ by 
\begin{equation}
\label{EquationDefinitionLinearFlow}
\rho_X\big(\phi_\alpha(t,p)\big)=\rho_X(x)+te_\alpha\mod\ZZ^2.
\end{equation}
Equation~\eqref{EquationDefinitionLinearFlow} determines $k_j+1$ trajectories starting at any conical point $p_j$, which may or may not be defined for any $t\geq0$, where the trajectory stops at $t=t_0$ if $\phi_\alpha(t_0,p_j)$ is also a conical point. Similarly we have $k_j+1$ trajectories ending in $p_j$. 
We call \emph{singular leaves} such trajectories. The flow $\phi_\alpha$ is a regular $\RR$-action outside singular leaves.
If $\alpha\in\RR\setminus\QQ$, then $\phi_\alpha$ is \emph{uniquely ergodic}, that is the Lebesgue measure of $X$ is the only invariant measure. This implies that any positive-infinite orbit is dense. On the other hand, if $\alpha\in\QQ$ then any infinite orbit is periodic, moreover periods take finitely many values.

\subsection{The \emph{Ornithorynque} origami}
\label{SectionOrnithorynqueOrigami}

Consider the set $\cQ:=\ZZ/3\ZZ\times\ZZ/2\ZZ\times\ZZ/2\ZZ$ and let $X_\cO$ be the origami defined by the pair 
$(h,v)$ of permutations of $\cQ$ given by 
%\begin{equation}
%\label{EquationHorizontalPermutationSurfaceX}
$$
h\left(
\begin{array}{c}
(i,0,0) \\
(i,0,1) \\
(i,1,0) \\
(i,1,1)
\end{array}
\right):=
\left(
\begin{array}{c}
(i+1,1,0) \\
(i-1,1,1) \\
(i,0,0) \\
(i,0,1)
\end{array}
\right)
$$
%\end{equation}
and
%\begin{equation}
%\label{EquationVerticalPermutationSurfaceX}
$$
v\left(
\begin{array}{c}
(i,0,0) \\
(i,0,1) \\
(i,1,0) \\
(i,1,1)
\end{array}
\right):=
\left(
\begin{array}{c}
(i-1,0,1) \\
(i,0,0) \\
(i+1,1,1) \\
(i,1,0)
\end{array}
\right).
$$
%\end{equation}

Figure~\ref{FigureOrnithorynque} represents the origami $X_\cO$. Half of the $24$ pairs of identified sides are represented by dotted lines. The other $12$ pairs are named by letters $A_i,B_i,C_i,D_i$ with $i\in\ZZ/3\ZZ$. There are three conical points $p_1,p_2,p_3$ with orders $k_1=k_2=k_3=2$, that is a conical angle $6\pi$ at each conical point. Figure~\ref{FigureOrnithorynque} shows $3$ big squares with size  $2\times 2$. The $12$ vertices of these big squares are identified to $p_1$, the $6$ middle points of the horizontal sides correspond to $p_2$ and the $6$ middle points of the vertical sides correspond to $p_3$. From the relation $2g-2=k_1+k_2+k_3$ we get that $X_\cO$ has genus $g=4$.

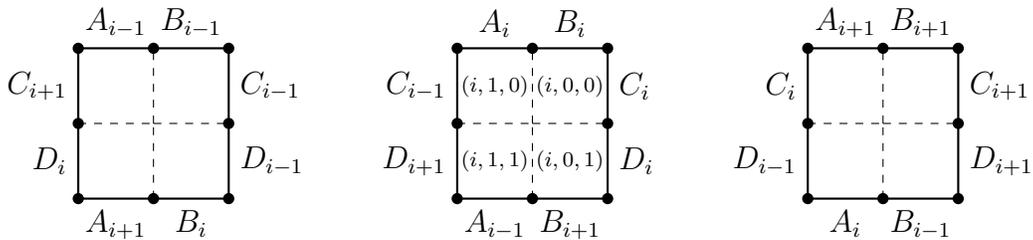
\begin{figure}[h]
\begin{center}  
\begin{tikzpicture}[scale=1]

%----------square(i-1)-------------

\draw[-,thick] (-1,-1) -- (0,-1) node[pos=0.5,below] {$A_{i+1}$};
\draw[-,thick] (0,-1) -- (1,-1) node[pos=0.5,below] {$B_i$};
\draw[-,dashed] (-1,0) -- (0,0) {};
\draw[-,dashed] (0,0) -- (1,0) {};
\draw[-,thick] (-1,1) -- (0,1) node[pos=0.5,above] {$A_{i-1}$};
\draw[-,thick] (0,1) -- (1,1) node[pos=0.5,above] {$B_{i-1}$};

\draw[-,thick] (-1,-1) -- (-1,0) node[pos=0.5,left] {$D_i$};
\draw[-,thick] (-1,0) -- (-1,1) node[pos=0.5,left] {$C_{i+1}$};
\draw[-,dashed] (0,-1) -- (0,0) {};
\draw[-,dashed] (0,0) -- (0,1) {};
\draw[-,thick] (1,-1) -- (1,0) node[pos=0.5,right] {$D_{i-1}$};
\draw[-,thick] (1,0) -- (1,1) node[pos=0.5,right] {$C_{i-1}$};

\node [circle,fill,inner sep=1.5pt] at (-1,-1) {};
\node [circle,fill,inner sep=1.5pt] at (0,-1) {};
\node [circle,fill,inner sep=1.5pt] at (1,-1) {};
\node [circle,fill,inner sep=1.5pt] at (-1,0) {};
\node [circle,fill,inner sep=1.5pt] at (1,0) {};
\node [circle,fill,inner sep=1.5pt] at (-1,1) {};
\node [circle,fill,inner sep=1.5pt] at (0,1) {};
\node [circle,fill,inner sep=1.5pt] at (1,1) {};

\end{tikzpicture}
\hspace{0.5 cm}
\begin{tikzpicture}[scale=1]

%----------square(i)-------------

\draw[-,thick] (-1,-1) -- (0,-1) node[pos=0.5,below] {$A_{i-1}$};
\draw[-,thick] (0,-1) -- (1,-1) node[pos=0.5,below] {$B_{i+1}$};
\draw[-,dashed] (-1,0) -- (0,0) {};
\draw[-,dashed] (0,0) -- (1,0) {};
\draw[-,thick] (-1,1) -- (0,1) node[pos=0.5,above] {$A_i$};
\draw[-,thick] (0,1) -- (1,1) node[pos=0.5,above] {$B_i$};

\draw[-,thick] (-1,-1) -- (-1,0) node[pos=0.5,left] {$D_{i+1}$};
\draw[-,thick] (-1,0) -- (-1,1) node[pos=0.5,left] {$C_{i-1}$};
\draw[-,dashed] (0,-1) -- (0,0) {};
\draw[-,dashed] (0,0) -- (0,1) {};
\draw[-,thick] (1,-1) -- (1,0) node[pos=0.5,right] {$D_i$};
\draw[-,thick] (1,0) -- (1,1) node[pos=0.5,right] {$C_i$};

\node [circle,fill,inner sep=1.5pt] at (-1,-1) {};
\node [circle,fill,inner sep=1.5pt] at (0,-1) {};
\node [circle,fill,inner sep=1.5pt] at (1,-1) {};
\node [circle,fill,inner sep=1.5pt] at (-1,0) {};
\node [circle,fill,inner sep=1.5pt] at (1,0) {};
\node [circle,fill,inner sep=1.5pt] at (-1,1) {};
\node [circle,fill,inner sep=1.5pt] at (0,1) {};
\node [circle,fill,inner sep=1.5pt] at (1,1) {};

\node at (-0.5,-0.5) {\tiny$(i,1,1)$};
\node at (-0.5,0.5) {\tiny$(i,1,0)$};
\node at (0.5,-0.5) {\tiny$(i,0,1)$};
\node at (0.5,0.5) {\tiny$(i,0,0)$};

\end{tikzpicture}
\hspace{0.5 cm}
\begin{tikzpicture}[scale=1]

%----------square(i+1)-------------

\draw[-,thick] (-1,-1) -- (0,-1) node[pos=0.5,below] {$A_i$};
\draw[-,thick] (0,-1) -- (1,-1) node[pos=0.5,below] {$B_{i-1}$};
\draw[-,dashed] (-1,0) -- (0,0) {};
\draw[-,dashed] (0,0) -- (1,0) {};
\draw[-,thick] (-1,1) -- (0,1) node[pos=0.5,above] {$A_{i+1}$};
\draw[-,thick] (0,1) -- (1,1) node[pos=0.5,above] {$B_{i+1}$};

\draw[-,thick] (-1,-1) -- (-1,0) node[pos=0.5,left] {$D_{i-1}$};
\draw[-,thick] (-1,0) -- (-1,1) node[pos=0.5,left] {$C_i$};
\draw[-,dashed] (0,-1) -- (0,0) {};
\draw[-,dashed] (0,0) -- (0,1) {};
\draw[-,thick] (1,-1) -- (1,0) node[pos=0.5,right] {$D_{i+1}$};
\draw[-,thick] (1,0) -- (1,1) node[pos=0.5,right] {$C_{i+1}$};

\node [circle,fill,inner sep=1.5pt] at (-1,-1) {};
\node [circle,fill,inner sep=1.5pt] at (0,-1) {};
\node [circle,fill,inner sep=1.5pt] at (1,-1) {};
\node [circle,fill,inner sep=1.5pt] at (-1,0) {};
\node [circle,fill,inner sep=1.5pt] at (1,0) {};
\node [circle,fill,inner sep=1.5pt] at (-1,1) {};
\node [circle,fill,inner sep=1.5pt] at (0,1) {};
\node [circle,fill,inner sep=1.5pt] at (1,1) {};

\end{tikzpicture}
\end{center}
\caption{The Ornithorynque origami $X_\cO$.}
\label{FigureOrnithorynque}
\end{figure}

The surface $X_\cO$ was discovered by Forni and Matheus in the preprint~\cite{ForniMatheusPreprint}, and then included in a larger family of surfaces in~\cite{ForniMatheusZorich}. After Delecroix and Weiss, the origami $X_\cO$ was named \emph{Ornithorynque} (french for Platypus), as a rare example of surface with totally degenerate \emph{Lyapunov spectrum}. Previously, in~\cite{ForniWollmilchsau}, Forni discovered the only other known example with such property, which is a genus $g=3$ surface $X_\cE$ called in german \emph{Eierlegende Wollmilchsau}. The surface $X_\cE$ was first introduced  in~\cite{HerrlichSchmithusen} and its name was given by Herrlich, M\"oller and Schmith\"usen, referring to its peculiar algebro-geometrical properties, which make $X_\cE$ a source of counterexamples in Teichm\"uller theory.

\subsection{Main statement}
\label{SectionMainStatement}

Recall that the \emph{diophantine type} of $\alpha\in\RR$ is 
$$
w(\alpha):=
\sup\left\{w>0:
\big|\alpha-p/q\big|<\frac{1}{q^{w+1}}\text{ for infinitely many }p/q\in\QQ
\right\},
$$
where as usual any fraction $p/q$ is written with co-prime $p$ and $q$. We always have $w(\alpha)\geq1$ by Dirichlet's Theorem. Moreover $w(\alpha)=1$ for almost any 
$\alpha$. Fix an origami $X$ and $\alpha\in\RR$. For any $p\in X$ and $r>0$, the time needed by the positive $\phi_\alpha$-orbit of $p$ to become $r$-dense is 
$$
T(X,\alpha,p,r):=
\sup\bigg\{\widetilde{p}\in X:
\inf\big\{t>r:\distance\big(\phi_\alpha(t,p),\widetilde{p}\big)<r\big\}
\bigg\},
$$ 
where $\distance(\cdot,\cdot)$ is the distance on $X$, which equals the euclidean distance on small enough discs in $X\setminus\{p_1,\dots,p_m\}$. Minimality implies that $T(X,\alpha,p,r)$ is defined for any $p$ outside singular leaves. In general, the scaling law of $T(X,\alpha,p,r)$ as $r\to 0$ has an irregular behaviour. Nevertheless it can be bounded from above by a power law $r^{-H}$, where the best exponent $H=H(X,\alpha,p)$, called \emph{hitting time}, is defined by 
$$
H(X,\alpha,p):=
\limsup_{r \to 0^+} \frac{\log T(X,\alpha,p,r)}{-\log r}.
$$

\begin{theorem}
\label{TheoremMainTheorem}
Let $X_\cO$ be the Ornithorynque origami. Then for any $\alpha$ irrational and any $p$ outside of singular leaves we have 
$$
H(X_\cO,\alpha,p)= w(\alpha).
$$
\end{theorem}

Theorem~\ref{TheoremGeneralCriterion} below proves the identity $H(X,\alpha,p)= w(\alpha)$ in a more general setting. The non-trivial inequality is 
$
H(X,\alpha,p)\leq w(\alpha)
$, 
which holds for any origami $X$ satisfying a specific intersection property. 
%(see the statement of Theorem~\ref{TheoremGeneralCriterion}). 
Proposition~\ref{PropositionMainIntersectionProperty} and Corollary~\ref{CorollaryMainIntersectionProperty} below show that $X_\cO$ satisfies such property. 
The same is true for the Eierlegende Wollmilchsau $X_\cE$ (\S~8.2 in~\cite{KimMarcheseMarmi}). 
Such property fails for any genus $2$ origami with one conical point (Lemma~6.5 in~\cite{KimMarcheseMarmi}). 
\emph{Cyclic covers} in~\cite{MatheusYoccoz} are a natural candidate for testing the assumption of Theorem~\ref{SectionProofMainTheorem} and thus proving the identity between diophantine type and hitting time. The easier inequality in Theorem~\ref{TheoremGeneralCriterion} and Theorem~\ref{TheoremMainTheorem} is implicit in~\cite{KimMarcheseMarmi}. We state it as follows (a proof is in \S~\ref{SectionProofLowerBoundHittingTime}).

\begin{lemma}
\label{LemmaLowerBoundHittingTime}
Let $X$ be any origami and $\alpha$ be an irrational slope. For any $p$ outside singular leaves we have 
$$
H(X,\alpha,p)\geq w(\alpha).
$$
\end{lemma}

For any origami $X$ and any $\alpha$ irrational, the function $p\mapsto H(X,\alpha,p)$ is invariant under $\phi_\alpha$ (Lemma~4.2 in~\cite{KimMarcheseMarmi}). Thus $H(X,\alpha,\cdot)$ is constant almost everywhere. Theorem~\ref{TheoremMainTheorem} was proved on the standard torus $X=\TT^2$ in~\cite{KimSeo}. 
Proposition~2.5 in~\cite{KimMarcheseMarmi} extends the same result to the Eierlegende Wollmilchsau $X_\cE$. On the other hand, for any origami $X$ with genus $g=2$ and an unique conical point of order $k_1=2$, Theorem~2.2 in~\cite{KimMarcheseMarmi} proves that for any $\lambda\in[1,2]$ there are directions $\alpha$ with 
$$
H(X,\alpha,p)=w(\alpha)^\lambda
\quad\text{ for almost any }\quad p\in X.
$$
For $X$ with the same topological data, we have $H(X,\alpha,p)\leq w(\alpha)^2$ for any $\alpha$ and any $p$ outside singular leaves (Theorem~2.1 in~\cite{KimMarcheseMarmi}). 
Proposition~4.6 in~\cite{KimMarcheseMarmi} proves that for any origami $X$ and $\alpha$ irrational we have
$$
\liminf_{r\to0}
\frac
{\log\big(\inf\big\{t>r:\distance\big(\phi_\alpha(t,p),\widetilde{p}\big)<r\big\}\big)}{-\log r}=1
\quad\text{ for almost any }\quad
p,\widetilde{p}\in X.
$$
Combining the last result and Theorem~\ref{TheoremMainTheorem}, and recalling that generically $w(\alpha)=1$, we get that for almost any $\alpha$ and almost any $p,\widetilde{p}$ in $X_\cO$ there exists the limit
$$
\lim_{r\to0}
\frac
{\log\big(\inf\big\{t>r:\distance\big(\phi_\alpha(t,p),\widetilde{p}\big)<r\big\}\big)}{-\log r}=1.
$$
The limit above was established for generic \emph{interval exchange transformations} in~\cite{KimMarmi}. Most results quoted from~\cite{KimMarcheseMarmi} are proved in the general setting of \emph{translation surfaces}.

\subsection*{Contents of this paper}

In \S~\ref{SectionBackground} we describe the action of $\sltwoz$ over the set of origamis, which fixes $X_\cO$. 
In \S~\ref{SectionIntersectionPropertyOrnithorynque} we state and prove 
Proposition~\ref{PropositionMainIntersectionProperty} and Corollary~\ref{CorollaryMainIntersectionProperty}, which establish that $X_\cO$ satisfies the intersection property in Theorem~\ref{TheoremGeneralCriterion}. 
In \S~\ref{SectionGeneralCriterion} we revise continued fractions in terms of $\sltwoz$ and use them as a renormalization tool to prove Theorem~\ref{TheoremGeneralCriterion}. 
The proof of Theorem~\ref{TheoremMainTheorem} is resumed in \S~\ref{SectionProofMainTheorem}. 
In \S~\ref{SectionProofLowerBoundHittingTime} we prove Lemma~\ref{LemmaLowerBoundHittingTime}.

\subsection*{Acknowledgements}

The author is grateful to D. H. Kim, S. Marmi and C. Matheus.

\section{Background}
\label{SectionBackground}

Let $\sltwoz$ be the group of $2\times2$ matrices $A$ with coefficients in $\ZZ$ and determinant 
$\det(A)=1$. In particular we consider the following elements
\begin{equation}
\label{EquationGeneratorsSL(2,Z)}
T:=
\begin{pmatrix}
1 & 1 \\ 
0 & 1
\end{pmatrix}
\quad\text{;}\quad
V:=
\begin{pmatrix}
1 & 0 \\ 
1 & 1
\end{pmatrix}
\quad\text{;}\quad
R:=
\begin{pmatrix}
0 & -1 \\ 
1 & 0
\end{pmatrix}.
\end{equation}
Any $A\in\sltwoz$ acts projectively on points $\alpha\in\RR\cup\{\infty\}$ by 
$$
A\cdot\alpha:=\frac{a\alpha+b}{c\alpha+d}
\quad\text{ where }\quad
A=
\begin{pmatrix}
a & b\\
c & d
\end{pmatrix}.
$$

\subsection{Action of $\sltwoz$}
\label{SectionActionSL(2,Z)}

Fix an origami $X$, defined by permutations $(h,v)$ of a finite set $\cQ$. Fix $A\in\sltwoz$ and consider the parallelogram $P:=A([0,1]^2)$. For $j\in\cQ$ the $j$-th copy $P_j:=\{j\}\times P$ has sides 
$$
\widetilde{l}_j:=\{j\}\times A(l)
\quad\text{;}\quad
\widetilde{r}_j:=\{j\}\times A(r)
\quad\text{;}\quad
\widetilde{b}_j:=\{j\}\times A(b)
\quad\text{;}\quad
\widetilde{t}_j:=\{j\}\times A(t),
$$
where the sides $l,r,b,t$ of $[0,1]^2$ are defined in \S~\ref{SectionOrigamisAndLinearFlows}. For any $j\in\cQ$, paste the side $\widetilde{r}_j$ of $P_j$ to the side $\widetilde{l}_{h(j)}$ of $P_{h(j)}$ and the side $\widetilde{t}_j$ of $P_j$ to the side $\widetilde{b}_{v(j)}$ of $P_{v(j)}$. Let $A\cdot X$ be the corresponding surface, which is compact, connected, orientable and without boundary. 
Moreover $A\cdot X$ is an origami, corresponding to a pair $(h^{(A)},v^{(A)})$ of permutations of $\cQ$. It is possible to see from the commutator $[h^{(A)},v^{(A)}]$ that $A\cdot X$ has the same number of conical points as $X$, with same orders $k_1,\dots,k_m$, and thus also the same genus (see~\cite{ForniMatheusIntroduction} for details). For the matrix $T$ in Equation~\eqref{EquationGeneratorsSL(2,Z)} we have 
%\begin{equation}
%\label{EquationActionOrigami(T)}
$$
h^{(T)}=h
\quad\text{ and }\quad
v^{(T)}=v\circ h^{-1},
$$
%\end{equation}
while for the matrix $V$ in Equation~\eqref{EquationGeneratorsSL(2,Z)} we have 
%\begin{equation}
%\label{EquationActionOrigami(V)}
$$
h^{(V)}=h\circ v^{-1}
\quad\text{ and }\quad
v^{(V)}=v.
$$
%\end{equation}
Since $T,V$ generate $\sltwoz$, we can compute $(h^{(A)},v^{(A)})$ from $(h,v)$ for any $A\in\sltwoz$.

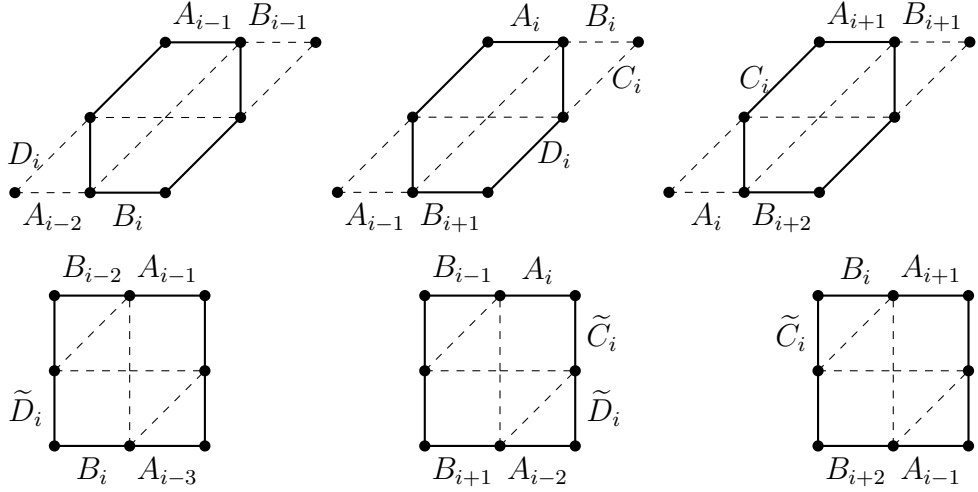
\begin{figure}[h]
\begin{center}

%--------T--of--squares-----------------

\begin{tikzpicture}[scale=1]

%--------T--of--square(i-1)-------------

\draw[-,dashed] (-2,-1) -- (-1,-1) node[pos=0.5,below] {$A_{i-2}$};
\draw[-,thick] (-1,-1) -- (0,-1) node[pos=0.5,below] {$B_i$};
\draw[-,dashed] (-1,0) -- (0,0) {};
\draw[-,dashed] (0,0) -- (1,0) {};
\draw[-,thick] (0,1) -- (1,1) node[pos=0.5,above] {$A_{i-1}$};
\draw[-,dashed] (1,1) -- (2,1) node[pos=0.5,above] {$B_{i-1}$};

\draw[-,dashed] (-2,-1) -- (-1,0) node[pos=0.5,left] {$D_i$};
\draw[-,thick] (-1,0) -- (0,1) {};
\draw[-,dashed] (-1,-1) -- (0,0) {};
\draw[-,dashed] (0,0) -- (1,1) {};
\draw[-,thick] (0,-1) -- (1,0) {};
\draw[-,dashed] (1,0) -- (2,1) node[pos=0.5,left] {};

\draw[-,thick] (-1,-1) -- (-1,0)  {};
\draw[-,thick] (1,0) -- (1,1) {};

\node [circle,fill,inner sep=1.5pt] at (-2,-1) {};
\node [circle,fill,inner sep=1.5pt] at (-1,-1) {};
\node [circle,fill,inner sep=1.5pt] at (0,-1) {};
\node [circle,fill,inner sep=1.5pt] at (-1,0) {};
\node [circle,fill,inner sep=1.5pt] at (1,0) {};
\node [circle,fill,inner sep=1.5pt] at (0,1) {};
\node [circle,fill,inner sep=1.5pt] at (1,1) {};
\node [circle,fill,inner sep=1.5pt] at (2,1) {};

\end{tikzpicture}
\begin{tikzpicture}[scale=1]

%-------T--of--square(i)-------------

\draw[-,dashed] (-2,-1) -- (-1,-1) node[pos=0.5,below] {$A_{i-1}$};
\draw[-,thick] (-1,-1) -- (0,-1) node[pos=0.5,below] {$B_{i+1}$};
\draw[-,dashed] (-1,0) -- (0,0) {};
\draw[-,dashed] (0,0) -- (1,0) {};
\draw[-,thick] (0,1) -- (1,1) node[pos=0.5,above] {$A_i$};
\draw[-,dashed] (1,1) -- (2,1) node[pos=0.5,above] {$B_i$};

\draw[-,dashed] (-2,-1) -- (-1,0) {};
\draw[-,thick] (-1,0) -- (0,1) {};
\draw[-,dashed] (-1,-1) -- (0,0) {};
\draw[-,dashed] (0,0) -- (1,1) {};
\draw[-,thick] (0,-1) -- (1,0) node[pos=0.5,right] {$D_i$};
\draw[-,dashed] (1,0) -- (2,1) node[pos=0.5,right] {$C_i$};

\draw[-,thick] (-1,-1) -- (-1,0)  {};
\draw[-,thick] (1,0) -- (1,1) {};

\node [circle,fill,inner sep=1.5pt] at (-2,-1) {};
\node [circle,fill,inner sep=1.5pt] at (-1,-1) {};
\node [circle,fill,inner sep=1.5pt] at (0,-1) {};
\node [circle,fill,inner sep=1.5pt] at (-1,0) {};
\node [circle,fill,inner sep=1.5pt] at (1,0) {};
\node [circle,fill,inner sep=1.5pt] at (0,1) {};
\node [circle,fill,inner sep=1.5pt] at (1,1) {};
\node [circle,fill,inner sep=1.5pt] at (2,1) {};

\end{tikzpicture}
\begin{tikzpicture}[scale=1]

%------T--of--square(i+1)-------------

\draw[-,dashed] (-2,-1) -- (-1,-1) node[pos=0.5,below] {$A_i$};
\draw[-,thick] (-1,-1) -- (0,-1) node[pos=0.5,below] {$B_{i+2}$};
\draw[-,dashed] (-1,0) -- (0,0) {};
\draw[-,dashed] (0,0) -- (1,0) {};
\draw[-,thick] (0,1) -- (1,1) node[pos=0.5,above] {$A_{i+1}$};
\draw[-,dashed] (1,1) -- (2,1) node[pos=0.5,above] {$B_{i+1}$};

\draw[-,dashed] (-2,-1) -- (-1,0)  {};
\draw[-,thick] (-1,0) -- (0,1) node[pos=0.5,left] {$C_i$};
\draw[-,dashed] (-1,-1) -- (0,0) {};
\draw[-,dashed] (0,0) -- (1,1) {};
\draw[-,thick] (0,-1) -- (1,0) {};
\draw[-,dashed] (1,0) -- (2,1) {};

\draw[-,thick] (-1,-1) -- (-1,0)  {};
\draw[-,thick] (1,0) -- (1,1) {};

\node [circle,fill,inner sep=1.5pt] at (-2,-1) {};
\node [circle,fill,inner sep=1.5pt] at (-1,-1) {};
\node [circle,fill,inner sep=1.5pt] at (0,-1) {};
\node [circle,fill,inner sep=1.5pt] at (-1,0) {};
\node [circle,fill,inner sep=1.5pt] at (1,0) {};
\node [circle,fill,inner sep=1.5pt] at (0,1) {};
\node [circle,fill,inner sep=1.5pt] at (1,1) {};
\node [circle,fill,inner sep=1.5pt] at (2,1) {};

\end{tikzpicture}

%----------new--squares-----------

\begin{tikzpicture}[scale=1]

%-------new--square(i-1)-------------

\draw[-,thick] (-1,-1) -- (0,-1) node[pos=0.5,below] {$B_i$};
\draw[-,thick] (0,-1) -- (1,-1) node[pos=0.5,below] {$A_{i-3}$};
\draw[-,dashed] (-1,0) -- (0,0) {};
\draw[-,dashed] (0,0) -- (1,0) {};
\draw[-,thick] (-1,1) -- (0,1) node[pos=0.5,above] {$B_{i-2}$};
\draw[-,thick] (0,1) -- (1,1) node[pos=0.5,above] {$A_{i-1}$};

\draw[-,thick] (-1,-1) -- (-1,0) node[pos=0.5,left] {$\widetilde{D}_i$};
\draw[-,thick] (-1,0) -- (-1,1) {};
\draw[-,dashed] (0,-1) -- (0,0) {};
\draw[-,dashed] (0,0) -- (0,1) {};
\draw[-,thick] (1,-1) -- (1,0) {};
\draw[-,thick] (1,0) -- (1,1) {};

\draw[-,dashed] (-1,0) -- (0,1)  {};
\draw[-,dashed] (0,-1) -- (1,0) {};

\node [circle,fill,inner sep=1.5pt] at (-1,-1) {};
\node [circle,fill,inner sep=1.5pt] at (0,-1) {};
\node [circle,fill,inner sep=1.5pt] at (1,-1) {};
\node [circle,fill,inner sep=1.5pt] at (-1,0) {};
\node [circle,fill,inner sep=1.5pt] at (1,0) {};
\node [circle,fill,inner sep=1.5pt] at (-1,1) {};
\node [circle,fill,inner sep=1.5pt] at (0,1) {};
\node [circle,fill,inner sep=1.5pt] at (1,1) {};

\end{tikzpicture}
\hspace{2.5 cm}
\begin{tikzpicture}[scale=1]

%--------new--square(i)-------------

\draw[-,thick] (-1,-1) -- (0,-1) node[pos=0.5,below] {$B_{i+1}$};
\draw[-,thick] (0,-1) -- (1,-1) node[pos=0.5,below] {$A_{i-2}$};
\draw[-,dashed] (-1,0) -- (0,0) {};
\draw[-,dashed] (0,0) -- (1,0) {};
\draw[-,thick] (-1,1) -- (0,1) node[pos=0.5,above] {$B_{i-1}$};
\draw[-,thick] (0,1) -- (1,1) node[pos=0.5,above] {$A_i$};

\draw[-,thick] (-1,-1) -- (-1,0) {};
\draw[-,thick] (-1,0) -- (-1,1) {};
\draw[-,dashed] (0,-1) -- (0,0) {};
\draw[-,dashed] (0,0) -- (0,1) {};
\draw[-,thick] (1,-1) -- (1,0) node[pos=0.5,right] {$\widetilde{D}_i$};
\draw[-,thick] (1,0) -- (1,1) node[pos=0.5,right] {$\widetilde{C}_i$};

\draw[-,dashed] (-1,0) -- (0,1)  {};
\draw[-,dashed] (0,-1) -- (1,0) {};

\node [circle,fill,inner sep=1.5pt] at (-1,-1) {};
\node [circle,fill,inner sep=1.5pt] at (0,-1) {};
\node [circle,fill,inner sep=1.5pt] at (1,-1) {};
\node [circle,fill,inner sep=1.5pt] at (-1,0) {};
\node [circle,fill,inner sep=1.5pt] at (1,0) {};
\node [circle,fill,inner sep=1.5pt] at (-1,1) {};
\node [circle,fill,inner sep=1.5pt] at (0,1) {};
\node [circle,fill,inner sep=1.5pt] at (1,1) {};

\end{tikzpicture}
\hspace{1.5 cm}
\begin{tikzpicture}[scale=1]

%--------new--square(i+1)-------------

\draw[-,thick] (-1,-1) -- (0,-1) node[pos=0.5,below] {$B_{i+2}$};
\draw[-,thick] (0,-1) -- (1,-1) node[pos=0.5,below] {$A_{i-1}$};
\draw[-,dashed] (-1,0) -- (0,0) {};
\draw[-,dashed] (0,0) -- (1,0) {};
\draw[-,thick] (-1,1) -- (0,1) node[pos=0.5,above] {$B_i$};
\draw[-,thick] (0,1) -- (1,1) node[pos=0.5,above] {$A_{i+1}$};

\draw[-,thick] (-1,-1) -- (-1,0) {};
\draw[-,thick] (-1,0) -- (-1,1) node[pos=0.5,left] {$\widetilde{C}_i$};
\draw[-,dashed] (0,-1) -- (0,0) {};
\draw[-,dashed] (0,0) -- (0,1) {};
\draw[-,thick] (1,-1) -- (1,0) {};
\draw[-,thick] (1,0) -- (1,1) {};

\draw[-,dashed] (-1,0) -- (0,1)  {};
\draw[-,dashed] (0,-1) -- (1,0) {};

\node [circle,fill,inner sep=1.5pt] at (-1,-1) {};
\node [circle,fill,inner sep=1.5pt] at (0,-1) {};
\node [circle,fill,inner sep=1.5pt] at (1,-1) {};
\node [circle,fill,inner sep=1.5pt] at (-1,0) {};
\node [circle,fill,inner sep=1.5pt] at (1,0) {};
\node [circle,fill,inner sep=1.5pt] at (-1,1) {};
\node [circle,fill,inner sep=1.5pt] at (0,1) {};
\node [circle,fill,inner sep=1.5pt] at (1,1) {};

%\node at (2,0) {$\dots$};

\end{tikzpicture}

\end{center}
\caption{Cut the dotted triangles in the above line and paste them along the sides $C_i,D_i$ for $i=0,1,2$, as in the line below. It follows $T\cdot X_\cO=X_\cO$.}
\label{FigureOrnithorynqueActionT}
\end{figure}

\begin{proposition}
\label{PropositionOrbitOrnithorynque}
We have $A\cdot X_\cO=X_\cO$ for any $A\in\sltwoz$.
\end{proposition}

\begin{proof}
Recall that $T,R$ in Equation~\eqref{EquationGeneratorsSL(2,Z)} generate $\sltwoz$.  Figure~\ref{FigureOrnithorynqueActionT} shows that $T\cdot X_\cO=X_\cO$, while it is clear from 
Figure~\ref{FigureOrnithorynque} that $R\cdot X_\cO=X_\cO$. See~\cite{ForniMatheusIntroduction} for more details.
\end{proof}

\subsection{Affine homeomorphisms}
\label{SectionAffineHomeomorphisms}

Fix an origami $X$ and $A\in\sltwoz$. 
For $j\in\cQ$, the affine maps $(j,x)\mapsto \big(j,A(x)\big)$ of $Q_j$ onto $P_j$ agree on glued sides, where we use the same notation of \S~\ref{SectionActionSL(2,Z)}. Therefore we have a globally defined homeomorphism 
\begin{equation}
\label{EquationAffineHomeomorphisms}
\psi_A:X\to A\cdot X
\end{equation}
sending $\{p_1,\dots,p_m\}$ bijectively onto the set of conical points of $A\cdot X$. 
Local inverses $\varphi:U\to X\setminus\{p_1,\dots,p_m\}$ of the covering $\rho_X$ in Equation~\eqref{EquationCoveringOverTorus}, defined over simply connected open sets $U\subset\TT^2$, give smooth charts for $X\setminus\{p_1,\dots,p_m\}$. Change of charts are indeed translations 
\footnote{Thus they are holomorphic, and one can extend them to an holomorphic atlas over the entire $X$, see~\cite{ForniMatheusIntroduction}.}. 
Similar translation charts exist on $A\cdot X$ (minus its conical points). In these translation charts $\psi_A$ is a diffeomorphism, 
which is locally affine. The \emph{linear part} $D\psi_A$ is the linear part of $\psi_A$ computed in translation charts. We have of course $D\psi_A=A$. The \emph{automorphisms group} $\autgroup(X)$ is the set of orientation preserving homeomorphisms $\psi:X\to X$ which preserve $\{p_1,\dots,p_m\}$ and are affine in translation charts, with $D\phi=\id$. In general $\autgroup(X)$ is non trivial, thus for a given $A\in\sltwoz$ there exist more than one $\psi_A$ as in Equation~\eqref{EquationAffineHomeomorphisms}. We have $\autgroup(X_\cO)\simeq \ZZ/3\ZZ$,
which acts by translation on the big $2\times 2$ squares in Figure~\ref{FigureOrnithorynque} (see \S~3.1 in~\cite{MatheusYoccoz}).

\section{The intersection property of $X_\cO$}
\label{SectionIntersectionPropertyOrnithorynque}

Let $X$ be any origami and $p_1,\dots,p_m$ be its conical points. Let $\rho_X:X\to\TT^2$ be the covering in Equation~\eqref{EquationCoveringOverTorus}. A \emph{straight segment} in $X$, or simply \emph{segment}, is a smooth path 
$S:(a,b)\to X\setminus\{p_1,\dots,p_m\}$ such that there exists a vector $v\in\RR^2$ with 
$$
\frac{d}{dt}\rho_X\big(S(t)\big)=v
\quad\text{ for any }\quad t\in(a,b).
$$
If $v=(x,y)\in\RR^2$, then the slope $\slope(S)\in\RR\cup\{\pm\infty\}$ of such $S$ is
$$
\slope(S):=\frac{x}{y}.
$$
The length $|S|$ of such segment is $|S|:=|b-a|\cdot\|v\|$, where $\|\cdot\|$ is the euclidean norm in 
$\RR^2$. Observe that segments do not contain conical points in their interior. Endpoints of straight segments can be conical points. A \emph{saddle connection} of the surface $X$ is a straight segment connecting conical points. Proposition~\ref{PropositionMainIntersectionProperty} is the main result in this section. Its proof is resumed in \S~\ref{SectionEndProofMainProposition} below, applying the constructions developed in \S~\ref{SectionPreliminaryIntersectionLemmas} and 
\S~\ref{SectionCuttingSequences}.

\begin{proposition}
\label{PropositionMainIntersectionProperty}
Let $X_\cO$ be the Ornithorynque origami. Fix segments $H,V$ in $X_\cO$ with $0<\slope(V)<1$ and $\slope(H)<-1$. If both segments have length $|H|,|V|\geq\sqrt{288}$ then 
$$
H\cap V\not=\emptyset.
$$
\end{proposition}

Let $S:\RR^2\to\RR^2$ acting by $S(x,y):=(-x,y)$. The same construction as in 
\S~\ref{SectionActionSL(2,Z)} gives a surface $S\cdot X_\cO$, obtained glueing copies 
$
\{j\}\times S\big([0,1]^2\big)
$ 
of the reflected square $S\big([0,1]^2\big)$, where $j\in\cQ$ and where identifications in 
$S\cdot X_\cO$ are induced by identifications in $X_\cO$. It is easy to see that indeed we have 
$S\cdot X_\cO=X_\cO$. As in \S~\ref{SectionAffineHomeomorphisms}, there exists an orientation reversing homeomorphism $f_S:X_\cO\to X_\cO$ with linear part $Df_S=S$. If $H,V$ are segments in $X_\cO$ with $-1<\slope(H)<0$ and $\slope(V)>1$, then $0<\slope\big(f_S(H)\big)<1$ and $\slope\big(f_S(V)\big)<-1$. 
Proposition~\ref{PropositionMainIntersectionProperty} implies directly the next Corollary.

\begin{corollary}
\label{CorollaryMainIntersectionProperty}
Let $X_\cO$ be the Ornithorynque origami. Fix segments $H,V$ in $X_\cO$ with $-1<\slope(H)<0$ and $\slope(V)>1$. If both segments have length $|H|,|V|\geq\sqrt{288}$ then 
$$
H\cap V\not=\emptyset.
$$
\end{corollary}

\subsection{Preliminary Lemmas}
\label{SectionPreliminaryIntersectionLemmas}

The proof of Lemma~\ref{LemmaIntersectionLines} below is left to the reader.

\begin{lemma}
\label{LemmaIntersectionLines}
Let $Q_1:=[0,1]^2$ and $Q_2:=[1,2]\times[0,1]$. Let $\ell_H$ and $\ell_V$ be two lines with 
$0<\slope(\ell_V)<1$ and $\slope(\ell_H)<-1$ and set $P:=\ell_H\cap\ell_V$. If both $\ell_V$ and $\ell_H$ intersect $\{1\}\times[0,1]$, then either $P\in Q_1$ or $P\in Q_2$.
\end{lemma}

\begin{lemma}
\label{LemmaSegmentsSameSquareIntersect}
Let $X$ be any origami labelled by a finite set $\cQ$. Fix a square $Q_j$ with $j\in\cQ$ and let $H$ and $V$ be segments in $X$ with $0<\slope(V)<1$ and $\slope(H)<-1$, such that both $H,V$ have endpoints in 
$
\bigcup_{l\not=j}\partial Q_l
$. 
If both $H\cap Q_k\not=\emptyset$ and $V\cap Q_k\not=\emptyset$ then $H\cap V\not=\emptyset$.
\end{lemma}

\begin{proof}
Let $\ell_H,\ell_V$ be lines as in Lemma~\ref{LemmaIntersectionLines} and $R$ be the matrix in 
Equation~\eqref{EquationGeneratorsSL(2,Z)}. The lines $R(\ell_H),R(\ell_V)$ satisfy the same assumption of Lemma~\ref{LemmaIntersectionLines}, with inverted roles. 
Thus Lemma~\ref{LemmaIntersectionLines} holds replacing $Q_2$ by any of the four unitary squares sharing a side with $Q_1$. Such extended version of Lemma~\ref{LemmaIntersectionLines} implies the statement, observing that if $H,V$ are segments as in the statement, then there must be a side of $Q_j$ intersecting both $H$ and $V$. 
\end{proof}

\subsection{Cutting sequences}
\label{SectionCuttingSequences}

Recall Figure~\ref{FigureOrnithorynque} and consider the twelve letters alphabet 
$$
\cA:=\{A_i,B_i,C_i,D_i:i=0,1,2\}.
$$ 
Geometrically, any element $\gamma\in\cA$ is a saddle connection of $X_\cO$. Symbolically, elements 
$\gamma\in\cA$ are letters which compose words $(\gamma_1,\dots,\gamma_n)$. Such words arise as \emph{cutting sequences} of straight segments $S$ in $X$. Fix a straight segment $S\subset X$, and abusing the notation denote $S:(0,1)\to X$ its parametrization with constant speed. Define recursively integers $k=1,\dots,n$ and instants $0\leq t_1<\dots<t_n\leq1$ by 
$$
\begin{array}{l}
t_1:=\min\{t\geq 0:\exists \gamma\in\cA:S(t)\in \gamma\}
\\
t_k:=\min\{t>t_{k-1}:\exists \gamma\in\cA:S(t)\in \gamma\}
\text{ for }k\geq2,
\end{array}
$$
where 
$
t_n=\max\{0\leq t\leq 1:\exists \gamma\in\cA:S(t)\in \gamma\}
$. 
Then define the cutting sequence 
$$
[S]:=(\gamma_1,\dots,\gamma_n)
$$
of $S$ as the word in the letters of $\cA$ such that 
$$
S(t_k)\in \gamma_k
\quad\text{ for }\quad
k=1,\dots,n.
$$

In the notation of \S~\ref{SectionOrnithorynqueOrigami}, for $i=0,1,2$ define the \emph{tile} 
$\cT_i\subset X$ by 
$$
\cT_i:=Q_{(i,1,0)}\cup Q_{(i,0,0)}\cup Q_{(i,1,1)}\cup Q_{(i,0,1)}.
$$

\begin{lemma}
\label{LemmaVerticalCrossesAllTiles}
Let $V$ be a segment with $0<\slope(V)<1$ and assume that its cutting sequence 
$[V]=(\gamma_1,\dots,\gamma_n)$ contains $n\geq6$ letters. 
Then $V\cap\cT_i\not=\emptyset$ for $i=0,1,2$.
\end{lemma}

\begin{proof}
Assume without loss of generality that $V$ does not cross the tile $\cT_0$. Then we have
$$
\gamma_k\not=A_0,B_0,C_0,D_0,C_2,D_1,A_2,B_1
\quad\text{for}\quad
k=1,\dots,n-1.
$$
Observing that 
$
\gamma_k=C_1\Rightarrow \gamma_{k+1}=A_2
$ 
we get
$$
\gamma_k\not=C_1
\quad\text{for}\quad
k=1,\dots,n-2.
$$
Since 
$
\gamma_k=B_2\Rightarrow \gamma_{k+1}\in\{D_1,C_1\} 
$ 
it follows 
$$
\gamma_k\not=B_2
\quad\text{for}\quad
k=1,\dots,n-3.
$$
Moreover we have 
$
\gamma_k=A_1\Rightarrow \gamma_{k+1}\in\{A_2,B_2,C_2\} 
$, 
therefore 
$$
\gamma_k\not=A_1
\quad\text{for}\quad
k=1,\dots,n-4.
$$
Finally 
$
\gamma_k=D_2\Rightarrow \gamma_{k+1}\in\{A_1,B_1\} 
$, 
which implies 
$$
\gamma_k\not=D_2
\quad\text{for}\quad
k=1,\dots,n-5.
$$
Since $n\geq6$, the conditions above imply that there is no value left for $\gamma_1$, which is absurd.
\end{proof}

\begin{lemma}
\label{LemmaHorizontalCrossesAllTiles}
Let $H$ be a segment with $\slope(H)<-1$ and assume that its cutting sequence 
$[H]=(\gamma_1,\dots,\gamma_n)$ contains $n\geq6$ letters. 
Then $H\cap\cT_i\not=\emptyset$ for $i=0,1,2$
\end{lemma}

\begin{proof}
The Lemma follows by an argument similar to Lemma~\ref{LemmaVerticalCrossesAllTiles}.  Alternatively 
consider $R$ in Equation~\eqref{EquationGeneratorsSL(2,Z)}, observe that $V:=R(H)$ satisfies the assumption of Lemma~\ref{LemmaVerticalCrossesAllTiles}, and recall $R\cdot X_\cO=X_\cO$.
\end{proof}

\begin{lemma}
\label{LemmaIntersectionTwoHorizontalCylinders}
Fix segments $H,V$ with $\slope(H)<-1$ and $0<\slope(V)<1$ and cutting sequences 
$
[H]=(\gamma_1,\dots,\gamma_n)
$ 
and 
$
[V]=(\nu_1,\dots,\nu_m)
$ 
with $n\geq4$ and $m\geq8$. Fix $i=0,1,2$ and assume that there exists $2\leq k\leq n-1$ with 
$$
\gamma_k\in\{C_{i+2},A_i\}
\quad\text{ and }\quad
\gamma_{k+1}\in\{B_{i+1},D_{i}\}.
$$
Then $H\cap V\not=\emptyset$.
\end{lemma}

\begin{proof}
Let $i\in\{0,1,2\}$ be as in the statement. Let $\widetilde{V}$ be the minimal subsegment of $V$ with cutting sequence 
$
[\widetilde{V}]=(\nu_2,\dots,\nu_{m-1})
$.
Lemma~\ref{LemmaVerticalCrossesAllTiles} implies $\widetilde{V}\cap\cT_i\not=\emptyset$. 
Let $\widetilde{H}$ be the minimal subsegment of $H$ with cutting sequence 
$
[\widetilde{H}]=(\gamma_2,\dots,\gamma_{n-1})
$. 
The assumption on $[H]$ implies that $\widetilde{H}$ intersects at least $3$ of the $4$ squares $Q_{(i,1,0)}$, $Q_{(i,0,0)}$, $Q_{(i,1,1)}$, $Q_{(i,0,1)}$ composing the tile $\cT_i$, where we recall the squares in an origami are closed and they overlap along the boundaries. Moreover the square missed by $\widetilde{H}$ can only be either $Q_{(i,1,1)}$ or $Q_{(i,0,0)}$, and it is crucial to observe that none of these two squares can contain $\widetilde{V}\cap\cT_i$. It follows that 
$
\widetilde{V}\cap Q\not=\emptyset
$ 
and $\widetilde{H}\cap Q\not=\emptyset$, where $Q$ is one of the four squares composing the tile $\cT_i$.  
Lemma~\ref{LemmaSegmentsSameSquareIntersect} gives $H\cap V\not=\emptyset$. 
\end{proof}

\begin{lemma}
\label{LemmaIntersectionOneHorizontalCylinder}
Fix segments $H,V$ with $\slope(H)<-1$ and $0<\slope(V)<1$ and cutting sequences 
$
[H]=(\gamma_1,\dots,\gamma_n)
$ 
and 
$
[V]=(\nu_1,\dots,\nu_m)
$ 
with $n\geq3$ and $m\geq7$. Fix $i=0,1,2$ and assume that 
\begin{equation}
\label{EquationLemmaIntersectionOneHorizontalCylinder}
(\gamma_k,\gamma_{k+1},\gamma_{k+2})=(C_{i+2},C_i,C_{i+1})
\quad\text{ or }\quad
(\gamma_k,\gamma_{k+1},\gamma_{k+2})=(D_i,D_{i+2},D_{i+1}).
\end{equation}
Then $H\cap V\not=\emptyset$.
\end{lemma}

\begin{proof}
Let $\widetilde{V}$ be the minimal subsegment of $V$ with cutting sequence 
$[\widetilde{V}]=(\nu_1,\dots,\nu_{m-1})$. Assume first 
$
(\gamma_k,\gamma_{k+1},\gamma_{k+2})=(C_{i+2},C_i,C_{i+1})
$. 
Let $\widetilde{H}_a,\widetilde{H}_b$ be respectively the minimal subsegments of $H$ with 
$[\widetilde{H}_a]=(C_{i+2},C_i)$ and $[\widetilde{H}_b]=(C_i,C_{i+1})$, so that in particular 
$\widetilde{H}_a\subset\cT_i$ and $\widetilde{H}_b\subset\cT_{i+1}$. 
We have $\widetilde{V}\cap \cT_{i+1}\not=\emptyset$ by 
Lemma~\ref{LemmaVerticalCrossesAllTiles}. If 
$
\widetilde{V}\cap\widetilde{H}_b=\emptyset
$ 
then there exists $k$ with $1\leq k\leq m-2$ with $(\nu_k,\nu_{k+1})=(B_{i+2},D_{i+1})$. Therefore the subsegment of $\widetilde{V}$ encoded by $(\nu_{k+1},\nu_{k+2})$ satisfies 
$\nu_{k+1}=D_{i+1}$ and $\nu_{k+2}\in\{A_i,B_i\}$, and this last property implies that such subsegment intersects $\widetilde{H}_a$. Thus $H\cap V\not=\emptyset$. Now assume 
$
(\gamma_k,\gamma_{k+1},\gamma_{k+2})=(D_i,D_{i+2},D_{i+1})
$. 
Let $\widetilde{H}_c,\widetilde{H}_d$ be respectively the minimal subsegments of $H$ with 
$[\widetilde{H}_c]=(D_i,D_{i+2})$ and $[\widetilde{H}_d]=(D_{i+2},D_{i+1})$, so that in particular 
$\widetilde{H}_c\subset\cT_{i+2}$ and $\widetilde{H}_d\subset\cT_{i+1}$. 
We have $\widetilde{V}\cap \cT_{i+1}\not=\emptyset$ by 
Lemma~\ref{LemmaVerticalCrossesAllTiles}. If 
$
\widetilde{V}\cap\widetilde{H}_d=\emptyset
$ 
then there exists $k$ with $1\leq k\leq m-2$ with $(\nu_k,\nu_{k+1})=(C_i,A_{i+1})$. Therefore the subsegment of $\widetilde{V}$ encoded by $(\nu_{k+1},\nu_{k+2})$ satisfies 
$\nu_{k+1}=A_{i+1}$ and $\nu_{k+2}\in\{A_{i+2},B_{i+2}\}$, and this last property implies that such subsegment intersects $\widetilde{H}_c$. Again it follows $H\cap V\not=\emptyset$. The Lemma is proved.
\end{proof}

\subsection{Proof of Proposition~\ref{PropositionMainIntersectionProperty}}
\label{SectionEndProofMainProposition}

Let $[H]=(\gamma_1,\dots,\gamma_n)$ and $[V]=(\nu_1,\dots,\nu_m)$ be the cutting sequences of $H,V$ respectively. Since $|H|,|V|\geq\sqrt{288}$, then we have both $n\geq12$ and $m\geq12$. 
Assume that the cutting sequence $[H]$ of $H$ does not satisfy 
Equation~\eqref{EquationLemmaIntersectionOneHorizontalCylinder} for any $i=0,1,2$. 
Then we must have $-6<\slope(H)<-1$. Since $n\geq12$, then $H$ satisfies the assumption of 
Lemma~\ref{LemmaIntersectionTwoHorizontalCylinders}.
Proposition~\ref{PropositionMainIntersectionProperty} follows.

\section{The general criterion}
\label{SectionGeneralCriterion}

%In this section we state and prove the following general criterion.

\begin{theorem}
\label{TheoremGeneralCriterion}
Let $X$ be an origami and assume that there exists a constant $\cK>0$ such that for any origami $Y\in\sltwoz\cdot X$ and any pair of segments $H,V\subset Y$ we have $H\cap V\not=\emptyset$ whenever they have length $|H|,|V|\geq \cK$ and satisfy
\begin{itemize}
\item 
either $\slope(H)<-1$ and $0<\slope(V)<1$
\item
or $-1<\slope(H)<0$ and $\slope(V)>1$.
\end{itemize}
Then $H(X,\alpha,p)=w(\alpha)$ for any $\alpha$ irrational and any $p$ outside $(X,\alpha)$-singular leaves. 
\end{theorem}

\subsection{Proof of Main Theorem~\ref{TheoremMainTheorem}}
\label{SectionProofMainTheorem}

Recall that $\sltwoz\cdot X_\cO=X_\cO$ by Proposition~\ref{PropositionOrbitOrnithorynque}. 
Theorem~\ref{TheoremMainTheorem} follows combining Theorem~\ref{TheoremGeneralCriterion} with Proposition~\ref{PropositionMainIntersectionProperty} and Corollary~\ref{CorollaryMainIntersectionProperty}

\subsection{Continued fractions}

Let $T,V$ be as in Equation~\eqref{EquationGeneratorsSL(2,Z)}. 
Consider positive integers $a_1,\dots,a_n$ and define $g(a_1,\dots,a_n)\in\sltwoz$ by
\begin{equation}
\label{EquationContinuedFractionSL(2,Z)}
g(a_1,\dots,a_n):=
\left\{
\begin{array}{c}
V^{a_1}\circ\dots\circ
V^{a_{n-1}}\circ T^{a_n}
\quad
\textrm{ for even }
n;
\\
V^{a_1}\circ\dots\circ
T^{a_{n-1}}\circ V^{a_n}
\quad
\textrm{ for odd }
n.
\end{array}
\right.
\end{equation}

Let $[\alpha]:=\max\{k\in\ZZ,k\leq \alpha\}$ be the \emph{integer part} and 
$\{\alpha\}:=\alpha-[\alpha]$ be the \emph{fractional part} of $\alpha\in\RR$, where 
$0\leq\{\alpha\}<1$. The \emph{Gauss map} $G:[0,1)\to[0,1)$ is defined by 
$$
G(\alpha):=\{\alpha^{-1}\}
\quad\text{ for }\quad \alpha\in[0,1).
$$
Any irrational $\alpha\in(0,1)$ admits an unique \emph{continued fraction expansion} 
\begin{equation}
\label{EquationContinuedFraction}
\alpha=[a_1,a_2,\dots]:=\cfrac{1}{a_1+\cfrac{1}{a_2+\dots}},
\end{equation}
where we set $\alpha_0:=\alpha$ and $\alpha_n:=G(\alpha_{n-1})$ for $n\geq1$, so that the $n$-th \emph{partial quotient} of $\alpha$ is given by 
$$
a_n:=\left[\frac{1}{\alpha_{n-1}}\right]
\quad
\textrm{ that is }
\quad
\frac{1}{\alpha_{n-1}}=a_n+\alpha_n.
$$
The $n$-th convergent $p_n/q_n:=[a_1,\dots,a_n]$ of $\alpha$ is obtained truncating Equation~\eqref{EquationContinuedFraction} to its $n$-th partial quotient $a_n$. We get 
\begin{equation}
\label{EquationMatrixContinuedFraction}
g(a_1,\dots,a_{2n-1})=
\begin{pmatrix}
p_{2n-1} & p_{2n-2} \\
q_{2n-1} & q_{2n-2}
\end{pmatrix}
\quad
\textrm{ and }
\quad
g(a_1,\dots,a_{2n})=
\begin{pmatrix}
p_{2n-1} & p_{2n} \\
q_{2n-1} & q_{2n}
\end{pmatrix}
\end{equation}
from the recursive relations $q_n=a_nq_{n-1}+q_{n-2}$ and $p_n=a_np_{n-1}+p_{n-2}$. 
Therefore
\begin{equation}
\label{EquationConvergentsProjectiveAction}
p_n/q_n=
\left\{
\begin{array}{c}
g(a_1,\dots,a_n)\cdot 0
\quad
\textrm{ for even }
n\\
g(a_1,\dots,a_n)\cdot \infty
\quad
\textrm{ for odd }
n.
\end{array}
\right.
\end{equation}

We have 
$
\alpha_n^{-1}=a_{n+1}+\alpha_{n+1}
\Leftrightarrow
\alpha_n=V^{a_{n+1}}\cdot \alpha_{n+1}^{-1}
\Leftrightarrow
\alpha_n^{-1}=T^{a_{n+1}}\cdot \alpha_{n+1}
$. 
Hence
\begin{equation}
\label{EquationActionSL(2,Z)SlopesIrrational}
\alpha=
g(a_1,\dots,a_{2k})\cdot \alpha_{2k}=
g(a_1,\dots,a_{2k},a_{2k+1})\cdot \frac{1}{\alpha_{2k+1}}
\quad\text{ for any }\quad k\in\NN.
\end{equation}

\subsection{Proof of Theorem~\ref{TheoremGeneralCriterion}}
\label{SectionProofTheoremGeneralCriterion}

Let $X$ be an origami as in Theorem~\ref{TheoremGeneralCriterion}. 
Write real numbers as $\alpha=a+\widetilde{\alpha}$, where $a:=[\alpha]$ and $\widetilde{\alpha}:=\{\alpha\}$ are the integer and fractional part respectively. Set $Y:=T^{-a}\cdot X$ and let $\psi:X\to Y$ be an affine homeomorphism with 
$D\psi=T^{-a}$ as in \S~\ref{SectionAffineHomeomorphisms}. We have $\kappa>0$ with 
$$
\phi_{\widetilde{\alpha}}\big(t,\psi(p)\big)
=
\psi\big(\phi_\alpha\big(\kappa t,p)\big)
\quad\text{ for any }\quad 
t\in\RR\text{ and }p\in X. 
$$ 
Thus 
$
H\big(Y,\widetilde{\alpha},\psi(p)\big)=H(X,\alpha,p)
$ 
and obviously any $Y\in\sltwoz\cdot X$ satisfies the same assumption as $X$. 
On the other hand $w(\widetilde{\alpha})=w(\alpha)$. 
Hence it is no loss of generality to consider $\alpha\in(0,1)$. We use Proposition~\ref{PropositionHittingSpecialTimes} below, whose proof is postponed to 
\S~\ref{SectionProofPropositionHittingSpecialTimes}.

\begin{proposition}
\label{PropositionHittingSpecialTimes}
Let $X$ and $\cK>0$ be an origami and a constant as in Theorem~\ref{TheoremGeneralCriterion}. 
Fix a slope  
$
\alpha=[a_1,a_2,\dots]\in(0,1)
$. 
For any $p\in X$ outside singular leaves and $n\in\NN$ we have
$$
T(X_\cO,\alpha,p,r_n)\leq 4\cK\cdot q_n
\quad
\textrm{ where }
\quad
r_n:=\frac{2(\cK+1)}{q_{n}}
$$
\end{proposition}

Set $w:=w(\alpha)$, so that $q_n\leq K\cdot q_{n-1}^w$ for some $K$ and all $n$. 
Fix $p\in X$ outside singular leaves. 
For any $r>0$ small enough consider $n$ with $r_{n-1}\leq r<r_{n}$. 
Proposition~\ref{PropositionHittingSpecialTimes} gives
\begin{align*}
\frac{\log T(X_\cO,\alpha,p,r)}{|\log r|}
&
\leq
\frac{\log T(X_\cO,\alpha,p,r_{n})}{|\log r_{n-1}|}
\leq
\frac{\log4\cK+\log q_n}{\log q_{n-1}-\log2(\cK+1)}
\\
&
\leq
\frac{\log4\cK+\log K+w\cdot \log q_{n-1}}{\log q_{n-1}-\log2(\cK+1)}\to w
\quad
\textrm{ for }
\quad
n\to+\infty.
\end{align*}
Hence $H(X,\alpha,p)\leq w$. Lemma~\ref{LemmaLowerBoundHittingTime} gives the other inequality. Theorem~\ref{TheoremMainTheorem} is proved. $\qed$

\subsection{Cylinder decompositions}

Let $X$ be any origami. A \emph{closed geodesic} is a straight segment $\sigma:[a,b]\to X$ with 
$\sigma(a)=\sigma(b)$, where such point is not conical. If $\rho_X$ is the covering in 
Equation~\eqref{EquationCoveringOverTorus}, then $\rho_X\circ\sigma$ is a closed geodesic in $\TT^2$ and must have rational slope. Thus $\slope(\sigma)\in\QQ\cup\{\infty\}$. 
Given any $p/q$ rational, a \emph{cylinder} in slope $p/q$ is a maximal open and connected subset 
$C\subset X$ foliated by closed geodesics $\sigma$ with same length and $\slope(\sigma)=p/q$. 
Set $\slope(C):=p/q$ and $|C|:=|\sigma|$, where $\sigma$ is any closed geodesic as above. The boundary 
$\partial C$ is union of saddle connections with slope $p/q$. 

Referring to Figure~\ref{FigureOrnithorynque}, the vertical path $\sigma:[0,6]\to X_\cO$ such that $\sigma(2i)$ is the middle point of $A_i$ for $i=0,1,2$ is an example of closed geodesic in $X_\cO$. 
We have $|\sigma|=6$ and $\slope(\sigma)=0$. The two vertical cylinders of $X_\cO$ are 
$$
C^{(+)}_0:=\bigcup_{i=0,1,2}Q_{(i,1,1)}\cup Q_{(i,1,0)}
\quad\text{ and }\quad
C^{(-)}_0:=\bigcup_{i=0,1,2}Q_{(i,0,1)}\cup Q_{(i,0,0)}
$$
We have a decomposition $X_\cO=C^{(+)}_0\cup C^{(-)}_0$, where the boundaries of the two cylinders are made by vertical saddle connections. 

Referring to~\cite{ForniMatheusIntroduction}, recall that any origami $X$ admits a \emph{cylinder decomposition} in the vertical slope $p/q=0$, with a number $l\geq1$ of cylinders $C^{(1)}_0,\dots,C^{(l)}_0$. For $i=1,\dots,l$ any cylinder has $\slope(C^{(i)}_0)=0$, integer length $L_i:=|C^{(i)}_0|$ and integer \emph{width} $W_i$, where $W_i$ is defined as the length of an horizontal segment in $C^{(i)}_0$ with endpoints in $\partial C^{(i)}_0$. Fix $p/q\in\QQ\cup\{\infty\}$, take $A\in\sltwoz$ with $A\cdot\infty=p/q$ and an origami $Y$ with $A\cdot Y=X$. Let $\psi:Y\to X$ be an affine homeomorphism with $D\psi=A$, as in \S~\ref{SectionAffineHomeomorphisms}. The vertical cylinder decomposition
$
Y=C^{(1)}_0\cup\dots\cup C^{(l)}_0
$ 
induces the cylinder decomposition of $X$ in slope $p/q$, that is  
\begin{equation}
\label{EquationCylinderDecomposition}
X=C^{(1)}_{p/q}\cup\dots\cup C^{(l)}_{p/q}
\quad\text{ where }\quad
C^{(i)}_{p/q}:=\psi(C^{(i)}_0)
\quad\text{ for }\quad
i=1,\dots,l.
\end{equation}

\begin{lemma}
\label{LemmaLengthTransversalSegment}
Consider an origami $X$, a slope $p/q\in\QQ\cup\{\infty\}$ and the decomposition in Equation~\eqref{EquationCylinderDecomposition}. Let $H$ be a segment in $X$ crossing the cylinders $C^{j_1}_{p/q},\dots,C^{j_n}_{p/q}$. We have 
$$
|H|\leq 
\frac{W_{j_1}+\dots+W_{j_n}}
{\sqrt{q^2+p^2}\cos\big|\arctan\big(\slope(H)\big)-\arctan(-q/p)\big|}.
$$ 
\end{lemma}

\begin{proof}
Any cylinder in Equation~\eqref{EquationCylinderDecomposition} has length 
$|C^{(j)}_{p/q}|=L_j\sqrt{q^2+p^2}$ and euclidean area $L_jW_j$. 
Let $\widetilde{H}_j\subset C^{(j)}_{p/q}$ be a segment with endpoints in $\partial C^{(j)}_{p/q}$. 
If $\slope(\widetilde{H}_j)=-q/p$, which is orthogonal to $p/q$, then 
$|\widetilde{H}_j|=W_j(q^2+p^2)^{-1/2}$. If $\widetilde{H}_j$ has a different slope, then its length increases by the inverse of the cosinus of the angle between $\slope(\widetilde{H}_j)$ and $-q/p$. 
The segment $H$ is union of $n$ segments 
$\widetilde{H}_{j_1},\dots,\widetilde{H}_{j_n}$ as above. The Lemma follows.
\end{proof}

\subsection{Proof of Proposition~\ref{PropositionHittingSpecialTimes}}
\label{SectionProofPropositionHittingSpecialTimes}

Let $X$ be an origami as in Theorem~\ref{TheoremGeneralCriterion} and $\alpha=[a_1,a_2,\dots]$ irrational. Fix any two points $p,\widetilde{p}$ in $X$, with $p$ outside $(X,\alpha)$-singular leaves. 

Consider first the case $n=2k$. Set 
$
A:=g(a_1,\dots,a_{2k})
$ 
and let $X_k\in\sltwoz\cdot X$ be the surface with $A\cdot X_k=X$. 
Let $\psi:X_k\to X$ be an affine homeomorphism with $D\psi=A$, as in \S~\ref{SectionAffineHomeomorphisms}. 
Set $\alpha_{2k}:=A^{-1}\cdot\alpha$ and $p_{2k}/q_{2k}:=A\cdot 0$ as in 
Equations~\eqref{EquationActionSL(2,Z)SlopesIrrational} and~\eqref{EquationConvergentsProjectiveAction}. 
Recalling Equation~\eqref{EquationCylinderDecomposition}, 
for $i=1,\dots,l$ let $C^{(i)}_0$ be the cylinder in the decomposition of $X_k$ in vertical slope 
$p/q=0$. Let $W_i$ be the width of  $C^{(i)}_0$. The cylinder decomposition of $X$ in slope $p_{2k}/q_{2k}$ is 
$X=C^{(1)}_{p/q}\cup\dots\cup C^{(1)}_{p/q}$, where $C^{(i)}_{p/q}:=\psi(C^{(i)}_0)$. 
Consider $\beta$ irrational such that 
$$
\left\{
\begin{array}{l}
A^{-1}\cdot\beta <-1\\
\cos\big|\arctan(\beta)-\arctan(-q_{2k}/p_{2k})\big|>1/2.
\end{array}
\right.
$$
The slope $\widetilde{\beta}=-q_{2k}/p_{2k}$ is orthogonal to $p_{2k}/q_{2k}$ and satisfies the first condition above, indeed recalling 
Equation~\eqref{EquationMatrixContinuedFraction} we have 
$$
A^{-1}\cdot\frac{-q_{2k}}{p_{2k}}=
\begin{pmatrix}
q_{2k} & -p_{2k} \\
-q_{2k-1} & p_{2k-1}
\end{pmatrix}
\cdot\frac{-q_{2k}}{p_{2k}}=
\frac{-(q_{2k}^2+p_{2k}^2)}{q_{2k}q_{2k-1}+p_{2k}p_{2k-1}}
<-a_{2k}<-1.
$$
The same condition is satisfied by some irrational slope $\beta$ close to $\widetilde{\beta}$, by continuity of the projective action of $A$. The second condition on $\beta$ is easily satisfied. 

Let $\widetilde{H}\subset X$ be a straight segment passing through $\widetilde{p}$ with 
$\slope(\widetilde{H})=\beta$. Consider the segment 
$H:=\psi^{-1}(\widetilde{H})\subset X_k$. We have $\slope(H)=A^{-1}\cdot\beta$, which is irrational since $\beta$ is irrational. Since $H$ has irrational slope, it is not a subsegment of a saddle connection of $X_k$. Therefore $H$ can be extended along the slope $A^{-1}\cdot\beta$ and we can assume that it has length $|H|=\cK$. If $H$ crosses the vertical cylinders $C^{(j_1)}_0,\dots,C^{(j_n)}_0$ of $X_k$, then we have $W_{j_1}+\dots+W_{j_n}\leq\cK+1$. 
The second condition on $\beta$ and Lemma~\ref{LemmaLengthTransversalSegment} imply  
$$
|\widetilde{H}|\leq
\frac{2(\cK+1)}{\sqrt{q_{2k}^2+p_{2k}^2}}\leq
\frac{2(\cK+1)}{q_{2k}}=r_{2k}.
$$

Since $p$ does not belong to any $(X,\alpha)$-singular leaf, then $\psi^{-1}(p)$ does not belong to any $(X_k,\alpha_{2k})$-singular leaf and it has infinite positive orbit. 
Let $V$ be a segment in $X_k$ which has an endpoint in $\psi^{-1}(p)$, with $\slope(V)=\alpha_{2k}$, length $|V|=\cK$. By assumption $V$ intersects $H$. In other words, we have $t>0$ with
$$
\phi_{\alpha_{2k}}\big(t,\psi^{-1}(p)\big)\in H
\quad
\textrm{ and }
\quad
0\leq t\leq |V|,
$$ 
Consider $T>0$ such that 
$
\psi\circ \phi_{\alpha_{2k}}(t,\cdot)=\phi_\alpha(T,\cdot)\circ \psi
$, 
so that we have 
$$
\phi_\alpha(T,p)=
\phi_\alpha\big(T,\psi \big(\psi^{-1}(p)\big)\big)=
\psi\big( \phi_{\alpha_{2k}}\big(t,\psi^{-1}(p)\big)\big)
\in \psi(H)=\widetilde{H}.
$$
Both $\widetilde{p}$ and $\phi_\alpha(T,p)$ belong to $\widetilde{H}$. Hence
$
\big|
\phi_\alpha(T,p)-\widetilde{p}
\big|
\leq 
|\widetilde{H}|
\leq 
r_{2k}
$. 
We have 
$$
T\leq |\psi(V)|\leq 
\|A\|\cdot|V|\leq 
(p_{2k}+q_{2k}+p_{2k-1}+q_{2k-1})\cdot|V|\leq
4\cK\cdot q_{2k}.
$$
Since $\widetilde{p}$ is arbitrary, we get 
$$
T(X,\alpha,p,r_{2k})\leq 4\cK\cdot q_{2k}.
$$

Now consider the case $n=2k-1$. Set 
$
A:=g(a_1,\dots,a_{2k-2},a_{2k-1})
$ 
and let $X_k\in\sltwoz\cdot X$ be the surface with $A\cdot X_k=X$. 
Let $\psi:X_k\to X$ be an affine homeomorphism with $D\psi=A$. 
Let $\alpha_{2k-1}$ be the slope related to $\alpha$ by Equation~\eqref{EquationActionSL(2,Z)SlopesIrrational}, that is 
$
\alpha=A\cdot \alpha_{2k-1}^{-1}
$. 
We have $A\cdot \infty=p_{2k-1}/q_{2k-1}$ by 
Equation~\eqref{EquationConvergentsProjectiveAction}. Moreover 
$
-1<A^{-1}\cdot(-q_{2k-1}/p_{2k-1})<0
$, 
indeed Equation~\eqref{EquationMatrixContinuedFraction} gives
$$
A^{-1}\cdot\frac{-q_{2k-1}}{p_{2k-1}}
=
\begin{pmatrix}
q_{2k-2} & -p_{2k-2} \\
-q_{2k-1} & p_{2k-1}
\end{pmatrix}
\cdot\frac{-q_{2k-1}}{p_{2k-1}}
=
\frac{-q_{2k-2}q_{2k-1}-p_{2k-2}p_{2k-1}}{q_{2k-1}^2+p_{2k-1}^2}.
$$
Therefore we can chose an irrational slope $\beta$ such that 
$$
\left\{
\begin{array}{l}
-1 < A^{-1}\cdot\beta <0\\
\cos\big|\arctan(\beta)-\arctan(-q_{2k-1}/p_{2k-1})\big|>1/2.
\end{array}
\right.
$$
Let $\widetilde{H}\subset X$ be a segment passing through $\widetilde{p}$ with 
$\slope(\widetilde{H})=\beta$ such that $H:=\psi^{-1}(\widetilde{H})$ is a segment in $X_k$ with length $|H|=\cK$. Let $V\subset X_k$ be a segment having an endpoint in $\psi^{-1}(p)$, with 
$\slope(V)=1/\alpha_{2k-1}$ and length $|V|=\cK$. By assumption we have $H\cap V\not=\emptyset$. The remaining part of the argument is as in case $n=2k$ and is left to the reader.
Proposition~\ref{PropositionHittingSpecialTimes} is proved. $\qed$

\appendix

\section{Proof of Lemma~\ref{LemmaLowerBoundHittingTime}}
\label{SectionProofLowerBoundHittingTime}

Let $X,\alpha,p$ be as in the statement. For $0\leq w<1$ and any $r>0$ small enough, the $r$-neighbourhood of the orbit segment 
$
\{\phi_\alpha(t,p):0\leq t\leq r^{-w}\}
$ 
has area bounded by $2\cdot r^{1-w}$. Thus $H(X,\alpha,p)\geq1$. Lemma~\ref{LemmaLowerBoundHittingTime} is proved for those slopes with $w(\alpha)=1$. Now assume $w(\alpha)>1$ and take any $w$ with $1<w<w(\alpha)$. We can assume $0<\alpha<1$, as in \S~\ref{SectionProofTheoremGeneralCriterion}. 
Write $\alpha=[a_1,a_2,\dots]$. There exist infinitely many $n$ with 
\begin{equation}
\label{EquationAppendix(1)}
a_{n+1}\geq q_n^{w-1}.
\end{equation}
It is not a loss of generality to assume that all $n$ as above are even, that is $n=2k$ (otherwise repeat the proof replacing the vertical slope $p/q=0$ by the horizontal $p/q=\infty$). 
Modulo subsequences, assume that there exists $X_0$ in the orbit $\sltwoz\cdot X$ such that
$$
g(a_1,\dots,a_{2k})\cdot X_0=X
\quad\text{ for any }k.
$$
Recall Equation~\eqref{EquationCylinderDecomposition} and let 
$X_0=C^{(1)}_0\cup\dots\cup C^{(l)}_0$ be the cylinder decomposition of $X_0$ in vertical slope 
$p/q=0$, where any $C^{(i)}_0$ has width $W_i$ and length $L_i$. Let $\widetilde{p}\in X_0$ be a point in the boundary of some vertical cylinder and not on any $(X_0,\alpha_{2k})$-singular leaf. According to Equation~(6.9) in~\cite{KimMarcheseMarmi}, if  
$
\alpha_{2k}<\min_{1\leq i\leq l}L_i^{-1}
$ 
then there exists $i$ with 
\begin{equation}
\label{EquationAppendix(3)}
\phi_{\alpha_{2k}}(t,\widetilde{p})\in C^{(i)}_0
\quad\text{ for }\quad
0<t<W_i\cdot\frac{\sqrt{1+\alpha_{2k}^2}}{\alpha_{2k}}.
\end{equation}
Since $\alpha_{2k}=(a_{2k+1}+\alpha_{2k})^{-1}<<1$, then 
Equation~\eqref{EquationAppendix(3)} holds. Equation~\eqref{EquationAppendix(1)} gives
\begin{equation}
\label{EquationAppendix(4)}
W_i\cdot\frac{\sqrt{1+\alpha_{2k}^2}}{\alpha_{2k}}
\geq
\frac{W_i}{\alpha_{2k}}
\geq 
a_{2k+1}W_i 
\geq 
a_{2k+1}
\geq
q_{2k}^{w-1}.
\end{equation}
Set $r_0:=1/4$. Equation~\eqref{EquationAppendix(3)} and Equation~\eqref{EquationAppendix(4)} imply that  for any $\widetilde{p}\in X_0$ there exists a cylinder $C^{(i)}_0$ and vertical closed geodesic  
$\sigma\subset C^{(i)}_0$ such that 
\begin{equation}
\label{EquationAppendix(5)}
\phi_{\alpha_{2k}}(t,\widetilde{p})\not\in N(\sigma,r_0)
\quad\text{ for }\quad
0\leq t\leq q_{2k}^{w-1}/2,
\end{equation}
where $N(\sigma,r)$ is the $r$-neighbourhood of $\sigma$. 
Set $A:=g(a_1,\dots,a_{2k})$ and let $\psi:X_0\to X$ be an affine homomorphism with $D\psi=A$. 
Recall that $p_{2k}/q_{2k}=A\cdot 0$ and $\alpha=A\cdot \alpha_{2k}$. Moreover 
$
\psi\circ \phi_{\alpha_{2k}}(t,\cdot)=\phi_\alpha(\kappa t,\cdot)\circ \psi
$, 
where the stretching factor of $A$ on vectors with slope $\alpha_{2k}$ satisfies 
$\kappa>q_{2k}/\sqrt{2}$ (Equation~(6.11) in~\cite{KimMarcheseMarmi}). 
Equation~\eqref{EquationAppendix(5)} implies that for any $p\in X$ there exists a cylinder $C\subset X$ with $\slope(C)=p_{2k}/q_{2k}$ and a closed geodesic $\widetilde{\sigma}\subset C$ with 
$$
\phi_\alpha(t,p)\not\in N\big(\widetilde{\sigma},r_0\cdot(q_{2k}^2+p_{2k}^2)^{-1/2}\big)
\quad\text{ for }\quad
0\leq t\leq (q_{2k}^{w-1}/2)\cdot(q_{2k}/\sqrt{2}),
$$
where the size of the neighbourhood of $\widetilde{\sigma}$ is derived from 
Lemma~\ref{LemmaLengthTransversalSegment}. Since $\alpha<1$ and thus $p_{2k}<q_{2k}$, setting 
$r_k:=(q_{2k}\sqrt{32})^{-1}$ we obtain 
$
T(X,\alpha,p,r_k)\geq q_{2k}^w/\sqrt{8}
$ 
and thus 

$$
H(X,\alpha,p)
\geq 
\limsup_{k\to\infty}
\frac{\log T(X,\alpha,p,r_k)}{|\log r_k|}
\geq
\limsup_{k\to\infty}
\frac{w\log q_{2k}-\log\sqrt{8}}{\log q_{2k}+\log\sqrt{32}}=w.
$$
Therefore $H(X_\cO,\alpha,p)\geq w(\alpha)$ since $w<w(\alpha)$ is arbitrary. Lemma~\ref{LemmaLowerBoundHittingTime} is proved. $\qed$

\end{document}